\documentclass[a4paper,10pt]{article}
\usepackage{calc}
\usepackage[all]{xy}
\usepackage[centertags]{amsmath}
\usepackage{latexsym}
\usepackage{amsfonts}
\usepackage{graphicx}
\usepackage{tikz}
\usepackage{cases}
\usepackage{amssymb}
\usepackage{amsthm}
\usepackage{color}
\usepackage{fancyhdr}
\usepackage[dvips]{epsfig}
\usepackage{newlfont}
\usepackage[latin1]{inputenc}
\usepackage[latin1]{inputenc}
\usepackage[english,french]{babel}
\usepackage{graphicx}
\usepackage{t1enc}
\usepackage[english,french]{babel}
\usepackage{fancybox}
\usepackage{graphicx}
\usepackage{t1enc}
\usepackage[french2]{minitoc}
\usepackage{mathrsfs}
\usepackage{amsfonts}
\usepackage{wasysym}
\usepackage{hyperref}
\usepackage{rotating}
\allowdisplaybreaks
\numberwithin{equation}{section} 
\theoremstyle{plain} 

\newtheorem{theodefi}{Theorem Definition}[section]
\newtheorem{lem}[theodefi]{Lemma}
\newtheorem{rem}[theodefi]{Remark}

\newtheorem{prop}[theodefi]{Proposition}

\newtheorem{fact}[theodefi]{Fact}

\newtheorem{cor}[theodefi]{Corollary}

\newtheorem{nota}[theodefi]{Notation}

\newcommand{\po}{{\textbf{Poin}}}

\newcommand{\cro}{{\textbf{Cr}}}
\newcommand{\cri}{{\textbf{CRI}}}
\newcommand{\D}{{ \mathcal D}}
\newcommand{\eg}{{{\ell_1} }}
\newcommand{\ed}{{{\ell_2} }}

\newcommand{\Q}{{\mathbb Q}}

\newcommand{\R}{{\mathbb R}}
\newcommand{\Ce}{{\mathcal S}^1}
\newcommand{\N}{{\mathbb N}}
\newcommand{\Z}{{\mathbb Z}}

 \makeatletter
\newcommand{\thechapterwords}
{ \ifcase \thechapter\or 1\or 2\or 3\or 4\or 5\or
	6\or 7\or 8\or 9\or 10\or 11\fi}
\def\thickhrulefill{\leavevmode \leaders \hrule height 2ex \hfill \kern \z@}
\def\@makechapterhead#1{%
	\vspace*{15\p@}%
	{\parindent \z@ \centering \reset@font
		\thickhrulefill\quad
		\scshape  {\chapnumfont \@chapapp{}}{\chapnumfont \thechapterwords}
		\quad \thickhrulefill
		\par\nobreak
		\vspace*{15\p@}%
		\interlinepenalty\@M
		\hrule
		\vspace*{15\p@}%
		\huge {\bfseries  #1}\par\nobreak
		\par
		\vspace*{15\p@}%
		\hrule
		\vskip 15\p@
}}
\def\@makeschapterhead#1{%
	\vspace*{15\p@}%
	{\parindent \z@ \centering \reset@font
		\thickhrulefill
		\par\nobreak
		\vspace*{15\p@}%
		\interlinepenalty\@M
		\hrule
		\vspace*{15\p@}%
		\huge \bfseries #1\par\nobreak
		\par
		\vspace*{15\p@}%
		\hrule
		\vskip 30\p@
}}

\DeclareFixedFont{\chapnumfont}{T1}{phv}{b}{n}{20pt}
\DeclareFixedFont{\chapchapfont}{T1}{phv}{b}{n}{16pt}
\DeclareFixedFont{\chaptitfont}{T1}{phv}{b}{n}{24.88pt}
\def\@makechapterhead#1{%
	\vspace*{15\p@}%
	{\parindent \z@ \centering \reset@font
		\thickhrulefill\quad
		\scshape {\chaptitfont\color[rgb]{0.00,0.50,1.00}\@chapapp{}}
		{\chapnumfont \thechapterwords}
		\quad \thickhrulefill
		\par\nobreak
		\vspace*{15\p@}%
		\interlinepenalty\@M
		\hrule
		\vspace*{15\p@}%
		{\Large\bfseries #1}\par\nobreak
		\par
		\vspace*{15\p@}%
		\hrule
		\vskip 30\p@
}}%

\begin{document}
	\selectlanguage{english}
	
	\title{\textbf{Cherry Maps with Different Critical Exponents: Bifurcation of Geometry}}
	\author{ \large{TANGUE NDAWA Bertuel }\\	\normalsize{bertuelt@yahoo.fr}\\
		\normalsize{Institute of Mathematics  and Physical Sciences}\\\normalsize{National Advanced School of Engineering}\vspace{0.5cm}}
	\maketitle
	
	\selectlanguage{english}
	
	To the TIEMGNI DEFFO's family: the mother Carine and the children 
	Abigail, Joakim, and Helza after the death of Richard the father of the family.
\section*{Abstract}

We consider order preserving $C^3$ circle maps with a flat piece, irrational rotation number and critical exponents $(\ell_1, \ell_2)$.

We detect a change in the geometry of the system. For $(\ell_1, \ell_2)\in[1,2]^2$ the geometry is degenerate and it becomes bounded for  $(\ell_1,\ell_2)\in [2, \infty)^2\setminus \{(2,2)\} $. When  the rotation number is of  the form	 $[abab\cdots]$; for some $a,b\in\N^*$, the geometry is  bounded for  $(\ell_1, \ell_2)$  belonging above a curve defined on
$ ]1, +\infty [^2$. As a consequence we estimate the Hausdorff dimension of the non-wandering set $K_f=\Ce\setminus \bigcup _{i=0}^{\infty} f^{-i}(U)$. Precisely, the Hausdorff dimension of this set is
equal to zero when the geometry is degenerate and it is strictly positive when the geometry is bounded.

\paragraph{}

	\underline{Key words}: Circle map, Irrational rotation number, Flat piece, Critical exponent, Geometry and Hausdorff Dimension.
	
\underline{MSC-2010}: 37E10.

\underline{Acknowledgements}:
The author was partially supported by the Centre d'Excellence Africain en Science Math\'{e}matique set Applications (CEA-SMA).

I would sincerely thank Prof. M. Martens and Prof. Dr.  L. Palmisano for introducing me to the subject of this paper, his valuable advice, continuous encouragement and helpful discussions. I thank  Prof. Dr.  Carlos Ogouyandjou for having participated presentations related to this work.

	\section{Introduction}
We study a certain class of weakly order preserving, non-injective (on an interval exactly; called flat piece) circle maps which appear naturally in the study of Cherry flows on the two dimensions torus (see \cite{8, 11, 14, 15}), non-invertible continuous circle map (see \cite{7}) and of the dependence of the rotation interval on the parameter value for one-parameter families of continuous  circle maps (see \cite{16}). The dynamics of circle maps with a flat interval has been intensively explored in the past years, see \cite{4, 5,  7, 14, 17, 18}.

We discuss the geometry of the non-wandering set (set obtained by removing from the circle all pre-images of the flat piece). Where the geometry is concerned, we discover a dichotomy; which generalize the one found in \cite{5}. Some of our maps show a
"degenerate geometry", while others seem to be subject to the "bounded geometry".

Before we can explain more precisely our results, we introduce our class, adopt some notations and present basic lemmas.

\subsection{The class of functions}
We fix $\ell_1, \ell_2\geq1$ and we consider the class $\mathscr{L} $ of continuous circle maps 
$f$ of degree one for which an arc $U$ exists so that the following properties hold:
\begin{enumerate}
	\item The image of $U$ is one point.
	\item The restriction of $f$ to $\Ce\setminus\overline{U} $ is a $C^3$-diffeomorphism onto its image.
	\item Let $(a,b)$ be a preimage of $U$ under the projection of the real line to $\Ce$. On some right-sided neighborhood of $b$, $f$ can be represented as
	\begin{equation*}
	h_r((x-b)^{\ell_2});
	\end{equation*}
	where, $h_r$ is $C^3$-diffeomorphism  on a two-sided neighbourhood of $b$.
	Analogously, on a left-sided neighborhood of $a$, $f$ equals
	\begin{equation*}
	h_l((x-a)^{\ell_1});
	\end{equation*}
	In the following, we assume that $h_l(x)=h_r(x)=x$. In fact, it is possible to effect $C^3$ coordinate changes near $a$ and $b$ that will allow us to replace both  $h_l$ and $h_r$ by the identity function.
\end{enumerate}
Let $F$ be a lift of $f$ 
on the real line. The rotation number $\rho (f)$ of $f$ is  defined (independently of $x$ and $F$)  by 
\begin{equation*}
\rho (f):=\lim_{n\rightarrow\infty }\dfrac{F^n(x)-x}{n}(mod\:1).
\end{equation*}
Let  $(q_n)$ be the sequence of denominators of the convergents of $\rho (f)$ (irrational) defined recursively by $q_1=1$, $q_2=a_1$ and $q_{n+1}=a_nq_{n}+q_{n-1} $ for all $n\geq 3$; with, 
\begin{equation*}\label{i1}
\rho (f)=[a_0a_1\cdots ]:=a_0+\dfrac{1}{a_1+\dfrac{1}{a_2+\dfrac{1}{ \ddots}}}
\end{equation*} 

\subsubsection*{Additional Assumption} Let $f\in\mathscr{L} $. 	We say that $Sf$ (the Schwarzian derivative  of $f$) is negative if,
\begin{equation}\label{add assumpt 1}
Sf(x):= \dfrac{D^3f(x)}{Df(x)}-\frac{3}{2}\left(\dfrac{D^2f(x)}{Df(x)}\right)^2<0;\quad\forall\:x,\; Df(x)\neq 0;\tag{A1}
\end{equation}
With, $D^nf$ the $n^{th} $ derivative of $f$; for $n\in\N$. 

We will assume in the  proof of the first part	of Theorem 1 that $f$ has a negative negative Schwarzian derivative. 

\subsection{Notations and Definitions}
The fundamental notations are established in \cite{5} (p.2-3).	
Let $f\in\mathscr{L} $.
\begin{enumerate}
	\item For every $i\in\Z$, the writing $\underline{i} $ means $f^i(U)$.
	\item Let $I$ and $J$ be two intervals. $(I,J)$ is the interval between
	$I$ and $J$. $[I,J):=I\cup(I,J)$ and $(I,J]:=(I,J)\cup J$.
	$|I|$ is the length of the interval $I$, $|[I,J)|:=|I|+|(I,J)|$ and $|(I,J]|:=|J|+|(I,J)|$. We say that $I$ and $J$ are comparable when $|I|$ and $|J|$ are comparable. That means, there is $k>0$ such that, $\dfrac{1}{ k}|I|<|J|<k|I|$.
	\item 
	For any sequence $\Gamma_n$  and for any  real $d$, we adopt the writings:
	\begin{equation*}
	\Gamma_n^{d(\ell_1,\ell_2)}:=\begin{cases}\begin{array}{lcl}
	\Gamma_n^{d\ell_1} & \mbox{if}& n\equiv 0[2]\\
	\Gamma_n^{d\ell_2} &\mbox{if} & n\equiv 1[2]
	\end{array}
	\end{cases}  \Gamma_n^{d(\frac{1}{\ell_1},\frac{1}{\ell_2})}:=\begin{cases}\begin{array}{lcl}
	\Gamma_n^{d\frac{1}{\ell_1}} & \mbox{if}& n\equiv 0[2]\\
	\Gamma_n^{d\frac{1}{\ell_2}} &\mbox{if} & n\equiv 1[2]
	\end{array}
	\end{cases}
	\end{equation*}	
	
\end{enumerate}	

\subsection{ Discussion and statement of the results}
\subsubsection*{Scaling ratios}
The sequence
\begin{equation*}
\alpha_n:=\dfrac{|(\underline{-q_n}, \underline{0})|}{|[\underline{-q_n}
	, \underline{0})|}=\dfrac{|(f^{-q_n}(U), U)|}{|(f^{-q_n}(U), U)|+|f^{-q_n}(U)|}.
\end{equation*}
measure the geometry near a critical point. In fact, it serves as scaling relating the geometries of successive dynamic partitions. 

The geometry is said degenerate  when $\alpha_n$ goes to zero
and the geometry is bounded, when $\alpha_n$ is bounded away from zero.

The study of this geometry is parametrised by rotation number and critical exponents. In \cite{5}, for pairs  $(\ell, \ell)$; $ \ell>1$ and $\rho$ irrational number of bounded type (i.e. $\max_na_n<\infty$), the authors found a transition between degenerate geometry and bounded
geometry. In fact, they show under \textbf{(\ref{add assumpt 1})} that, if $ 1<\ell\leq2$ and $\rho\in\R\setminus\Q$, then the  geometry is degenerate and it is bounded (independently of \textbf{(\ref{add assumpt 1})}) if $ \ell>2$ and  $\rho$ is irrational number of bounded type. In \cite{4}, for $ \ell>1$, the author proved that the class of function $f$ of critical exponents $(1, \ell)$ or $(\ell, 1)$ have a degenerate geometry. In \cite{15}, the authors show that, for the maps in $\mathscr{L} $ with Fibonacci rotation number,  when the critical exponents $(\ell_1,\ell_2)$ belong in $(1,2)^2$, the geometry is degenerate. Let us note that, 
Differently from other previous works, information on the geometry of the system is obtained by the study of the asymptotic behaviour of the renormalization operator.

In the present paper, we consider the cases where the critical exponents  $(\ell_1,\ell_2)$ belong in a subdomain (containing the previous domains) of $[1,\infty)^2$; also, the rotation number is not necessarily Fibonacci type and the results do not depend on the renormalization operator as in \cite{9}.
We use the formalism presented in \cite{5} which is based on recursive inequalities analysis of $\alpha_{n} $. For technical reason, in the case of bounded geometry, we introduce the vector sequence $v_n:=(-\ln \alpha_{n}, -\ln \alpha_{n-1})$ and, the new recursive inequality is controlled by a $2\times2$ matrix. When the rotation number is bi-periodic ($\rho=[abab\cdots]; a,b\in\N$), the  $2\times2$ matrix has two eigenvalues (depending on rotation number and critical exponents) $\lambda_s\in(0,1)$ and $\lambda_u>0$. The equation $\lambda_u=\lambda_u((a,b); (\ell_1,\ell_2))=1$ defines a curve $\mathcal{C}_{\lambda_u=1}$ (presented above) which 
separates the $(\ell_1,\ell_2)$  plan into two components ${C}_{\lambda_u>1}$ (below the curve) and ${C}_{\lambda_u<1}$.

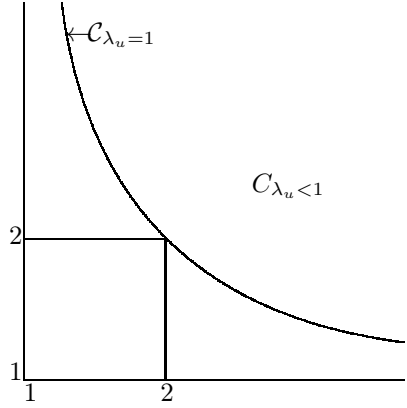
\begin{center}
	\setlength{\unitlength}{5cm}
	\begin{picture}(1, 1)
	\put(0,0){\line(0,1){1}}
	\put(0,0){\line(1,0){1}}
	\qbezier(0.1,1)(0.2,0.2)(1,0.1)
	\put(0.373,0){\line(0,1){0.373}}
	\put(0,0.373){\line(1,0){0.373}}
	\put(0.36,-0.057){2}
	\put(-0.038,0.36){2}
	\put(0,-0.057){1}
	\put(-0.038,0){1} 
	\put(0.11,0.9){$\leftarrow$} 
	\put(0.6,0.5){${C}_{\lambda_u<1}$}
	\put(0.169,0.9){$\mathcal{C}_{\lambda_u=1}$} 
	\end{picture}
\end{center}

	\paragraph{Main result}
	Let  $f\in \mathscr{L} $ with critical exponents $(\ell_1,\ell_2)$. Then,
	\begin{enumerate}	
		\item  the scaling ratio $\alpha_n$ goes to zero  when $(\ell_1,\ell_2)\in [1,2]^2$, $\rho\in\R\setminus \Q$ and \textbf{(\ref{add assumpt 1})} holds.
		\item  the scaling ratio $\alpha_n$ bounded away from zero  when $(\ell_1,\ell_2)\in [2, \infty)^2\setminus \{(2,2)\} $ and $\rho$ is bounded type. \item   the scaling ratio $\alpha_n$ bounded away from zero  when $(\ell_1, \ell_2)\in {C}_{\lambda_u<1}$ and $\rho$ is bi-periodic. 
	\end{enumerate}

\paragraph{Estimation of Hausdorff Dimension of the non-wandering set.}

In the symmetric case $(\ell_1=\ell_2=\ell)$, in  \cite{8},  for $f\in \mathscr{L} $ with bounded type rotation number,  the author  shows that, if   $\ell\in(1,2]$, then the Hausdorff dimension of the
non-wandering set is equal to zero and   that, 
if   $\ell>2$, then the Hausdorff dimension of the non-wandering set is strictly greater than zero. This result generalizes the one in \cite{19} where the author treats the maps in $\mathscr{L} $ with critical exponents $(1,1)$.  Let us note that, the results in  \cite{8} are more general (they only depend on geometry); precisely, if the rotation number is the bounded type, then, the Hausdorff Dimension of the non-wandering set is equal to zero when the geometry is degenerate and it is strictly greater than zero when the geometry is bounded; so we have the following result which is proved at the end of the paper. 
\begin{cor}
	Let  $f\in \mathscr{L} $ with critical exponent $(\ell_1,\ell_2)$ with bounded type rotation number. Then,
	\begin{enumerate}	
		\item   the Hausdorff dimension of the
		non-wandering set is equal to zero when $(\ell_1,\ell_2)\in [1,2]^2$ and \textbf{(\ref{add assumpt 1})} holds and for the pairs $(\ell, 1)$; $\ell>1$;
		\item the Hausdorff dimension is strictly bigger than zero  when $(\ell_1,\ell_2)\in [2, \infty)^2\setminus \{(2,2)\} $;
		\item  the Hausdorff dimension is strictly bigger than zero when $(\ell_1, \ell_2)\in {C}_{\lambda_u<1}$  and $\rho$ is bi-periodic.
	\end{enumerate}
\end{cor}	
The following remark will simplify statements and  proofs  of  results.	
\begin{rem}\label{critical exponent and  renormalization}
	Let us note that, our setting has some inherent symmetry. This will simplify statements and proof of our results. 
\end{rem}

\section{Tools}

\subsection{Cross-Ratio Inequalities }		
\begin{nota}	
	We denote by $\R_<^4$, the subset of  $\R^4$ defined by		
	\begin{equation*}
	\R_<^4:= \{ (x_1,x_2,x_3,x_4)\in\R^4,\; such\; that\;x_1<x_2<x_3<x_4 \}.
	\end{equation*}	
\end{nota}
The   following result can be found in \cite{1} (\textbf{Theorem 2}) 
\begin{prop}\label{CRI}
	The Cross-Ratio Inequality (\cri).  
	
	Let  $f\in \mathscr{L} $. Let $(a,b,c,d)\in \R_<^4$. 
	The   cross-ratio \cro\, is defined by 		
	\begin{equation*}
	\cro\:(a,b,c,d):=\dfrac{|b-a||d-c|}{|c-a||d-b|}
	\end{equation*}
	and the  cross-ratio \po\, is defined by
	\begin{equation*}
	\po\:(a,b,c,d):=\dfrac{|d-a||b-c|}{|c-a||d-b|}.
	\end{equation*}

	The distortion of the cross-ratio \cro\, and  cross-ratio \po\,  are given respectively by
	\begin{equation*}
	\D\cro\:(a,b,c,d):=\dfrac{\cro\:(f(a),f(b),f(c),f(d))}{\cro\:(a,b,c,d)}
	\end{equation*} 
	and
	\begin{equation*}
	\D\po\:(a,b,c,d):=\dfrac{\po\:(f(a),f(b),f(c),f(d))}{\po\:(a,b,c,d)}.
	\end{equation*}

	Let us consider a set of $n+1$ quadruples $\{a_i, b_i, c_i, d_i \} $ with the following properties:
	
	\begin{enumerate}
		\item Each point of the circle belongs to at most $k$ intervals $(a_i, d_i)$;
		\item The intervals $(b_i, c_i)$ do not intersect $U$.\label{hypothesis 2 of cri}
	\end{enumerate}
	Then there is $C_k,C'_k>0$ such that the following inequalities hold\,
	\begin{equation*}
	\prod_{i=0}^{n}\D\cro\:(a_i, b_i, c_i, d_i)\leq C_k
	\end{equation*} 
	and 
	\begin{equation*}
	\prod_{i=0}^{n}\D\po\:(a_i, b_i, c_i, d_i)\geq C'_k.
	\end{equation*} 
\end{prop}
Observe that, $\po+\cro=1.$ 
Thus, by the \textbf{1.5 Lemma} in \cite{2} we have the following result.
\begin{prop}\label{constant cross ratio inequality}	
	Let $f$ be a $C^3$ function such that the Schwarzian derivative is negative. Then,	$C'_k>1$; that is, $C_k<1$.	
\end{prop}

\begin{rem}
	Let $I$ and $J$ be two intervals finite and non-zero length such that,
	$\bar{I}\cap \bar{J}=\emptyset $. We assume that, $J$ is on the right of $I$ and we put $I:=[a,b]$ and $J:=[c,d]$, then
	
	\begin{equation*}
	\cro\:(I,J):=\dfrac{|I||J|}{|[I,J)||(I,J]|}=\cro\:(a,b,c,d)
	\end{equation*}
	and 	
	\begin{equation*}
	\po\:(I,J):=\dfrac{|(I,J)||[I,J]|}{|[I,J)||(I,J]|}=\po\:(a,b,c,d).
	\end{equation*}
\end{rem}
\begin{fact}\label{fact} 
	Let  $f\in \mathscr{L} $. Let $l(U)$ and $r(U)$ be the left and right endpoints of $U$ (the flat piece of $f$) respectively. There are a left-sided neighborhood $I^{l}$ of $l(U)$, a	right-sided neighborhood $I^{r}$ of $r(U)$ and  three positive constants $K_1, K_2, K_3$ such that the following holds 
	
	\begin{enumerate}
		\item If  $y\in I^{l_i} $ with  $l_1:=l,\;l_2:=r$, then
		\begin{equation*}
		\begin{array}{c}	
		K_{1}|l_i(U)-y|^{l_i}  \leq  |f(l_i(U))-f(y) | 	\leq  K_{2}|l_i(U)-y|^{l_i},\\
		K_{1}|l_i(U)-y|^{l_i-1}  \leq \dfrac{df}{dx}(y)  	\leq  K_{2}|l_i(U)-y|^{l_i-1}.
		\end{array}
		\end{equation*}	
		\item \label{rl4}
		If $y\in(x, z)\subset I^{l_i} $, with $z$ the closest point  to the flat interval $U$ then
		\begin{equation*}
		\dfrac{|f(x)-f(y)|}{|f(x)-f(z)|}\leq  K_3 \dfrac{|x-y|}{|x-z|}. 
		\end{equation*}.
	\end{enumerate}
\end{fact} 

The first part of  \textbf{Fact~\ref{fact}} implies that,
\begin{equation} \label{add assump 2}
f_{|I^{l_i}} \approx k_i x^{\ell_i};\;i=1,2. \quad \mbox{We assume that, }\quad k_1=k_2.\tag{A2}
\end{equation}
We need this assumption to  prove that $\alpha_n$ goes to zero (\textbf{Lemma~\ref{sigma go a away}} and \textbf{Lemma~\ref{use additional condition}}).	

\subsection{Basic Results}
\begin{prop}\label{partition of f}
	Let $ n\geq1$.
	\begin{enumerate}
		\item[$ \bullet $] The set of ``long'' intervals consists of the intervals
		\begin{equation*}		
		\mathcal{A}_n:=\{(\underline{i},\underline{q_{n}+i});\;0\leq i \leq q_{n+1}-1 \}.
		\end{equation*} 
		\item[$ \bullet $]  The set of ``short'' intervals consists of the intervals
		\begin{equation*}		
		\mathcal{B}_n:=\{(\underline{q_{n+1}+i},\underline{i});\;0\leq i \leq q_{n}-1 \}.
		\end{equation*}
	\end{enumerate}
	The set $\mathcal{P}_n:=\mathcal{A}_n\cup \mathcal{B}_n$ covers the circle modulo the end points and the flat piece and it is called the $n^{th}$ dynamical partition.
	
	The dynamical partition produced by the first $\underline{q_{n+1}+q_n} $ pre-images of $U$ is denoted  $\mathcal{P}^{n} $. It consists of 
	\begin{equation*}		
	\wp_n:=\{\underline{-i};\;0\leq i \leq q_{n+1}+q_{n}-1 \}
	\end{equation*}
	together with the gaps between these sets. As in the case of $\mathcal{P}_{n} $ there are two kinds of gaps, ``long'' and ``short'':
	\begin{enumerate}
		\item[$ \bullet $] The set of ``long'' intervals consists of the intervals
		\begin{equation*}		
		\mathcal{A}^{n}:=\{(\underline{-q_{n}-i},\underline{-i})=:I^{n}_i;\;0\leq i \leq q_{n+1}-1 \}.
		\end{equation*} 
		\item[$ \bullet $]  The set of ``short'' intervals consists of the intervals
		\begin{equation*}		
		\mathcal{B}^{n}:=\{(\underline{-i},\underline{-q_{n+1}-i})=:I^{n+1}_i;\;0\leq i \leq q_{n}-1 \}.
		\end{equation*}
	\end{enumerate}
\end{prop}

\begin{prop}\label{qn go to zero} 
	The sequence  $|( \underline{0},\underline{q_n})| $
	tends to zero  at least exponentially fast.
\end{prop}
\begin{prop}\label{preimage and gaps adjacent} 
	If $A$ is a pre-image of $U$ belonging to $\mathcal{P}^{n} $ and if $B$ is one of the gaps adjacent to $A$, then $|A|/|B|$ is bounded away from zero by a constant that does not depend on $n$, $A$ or $B$.
\end{prop}

\begin{lem}\label{sigma go a away}
	The  sequence 
	\begin{equation*}	
	f(\sigma_n)=	\dfrac{|( \underline{1},\underline{q_n+1})|}{|( \underline{q_{n-1}+1}, \underline{1})|}	
	\end{equation*}	
	is bounded.	\end{lem}
The proofs of these results can be found in \cite{5} (proof of the \textbf{Proposition 1}, \textbf{Proposition 2} and \textbf{Lemma 1.3}).

A proof of the following Proposition can be found in \cite{3} \textbf{theorem:3.1 p.285}).
\begin{prop}[Koebe principle]\label{koebe principle}
	Let	$ f\in \mathscr{L}$. For
	every $\varsigma,\,\alpha>0$, there exist a constant $\zeta(\varsigma,\alpha)>0$, such that,  the following holds. Let $T$ and $M\subset T$ be
	two intervals and let $S,\;D$ be the left and the right component
	of $T\setminus M$ and $n\in\N$. Suppose that:
	\begin{enumerate}
		\item $\sum_{i=0}^{n-1}f^i(T)<\varsigma $,
		\item $f^n: T\longrightarrow f^n(T)$ is a diffeomorphism,
		\item $\dfrac{|f^n(M)|}{|f^n(S)|}, \;\dfrac{|f^n(M)|}{|f^n(D)|}<
		\alpha.$
	\end{enumerate}
	
	Then,
	\begin{equation*}
	\dfrac{1}{\zeta(\varsigma,\alpha)}\leq\dfrac{	Df^n(x)}{	Df^n(y)}\leq\zeta(\varsigma,\alpha), \quad \forall x,y\in M;
	\end{equation*}
	that is,
	\begin{equation*}
	\dfrac{1}{\zeta(\varsigma,\alpha)}\cdot\dfrac{|A|}{|B|}\leq\dfrac{	f^n(A)}{	f^n(B)}\leq\zeta(\varsigma,\alpha)\cdot\dfrac{|A|}{|B|}, \quad \forall A, B\,\mbox{(intervals)}\subseteq M;
	\end{equation*}
	where,
	\begin{equation*}
	\zeta(\varsigma,\alpha)=\dfrac{1+ \alpha}{\alpha}e^{C\varsigma}
	\end{equation*}
	and $C\geq 0$ only depends on $f$.
\end{prop}

\begin{rem}\label{diffeo on an interval}
	Let $ f\in \mathscr{L} $. Given  $n>1$, $T$  and  $M$ as before. $f^{n}:T\longrightarrow f^{n}(T)$ is diffeomorphism if only if, for all $0\leq i\leq n-1$, $f^{i}(T)\cap \overline{U}=\emptyset$; where, $ \overline{U} $ 	designates the 	closure of $U$.  
\end{rem}

\section{ Proof of  results}
Let us put together (parameter) sequences which are frequently used in this section.
\begin{equation*}
\alpha_n=\dfrac{|(\underline{-q_n}, \underline{0})|}{|[\underline{-q_n}
	, \underline{0})|},\quad 
\sigma_n=	\dfrac{|( \underline{0},\underline{q_n})|}{|( \underline{q_{n-1}}, \underline{0}|},\quad s_n:=\dfrac{|[\underline{-q_{n-2}}, \underline{0}]|}{|\underline{0}|},\quad \tau_n:=\dfrac{|(\underline{0}, \underline{q_{n}}) |}{|(\underline{0}, \underline{q_{n-2}}) |}
\end{equation*}
and
\begin{equation*}
\beta_n(k)=\dfrac{|(\underline{-q_{n}+kq_{n-1}}, \underline{0})|}{|[\underline{-q_{n}+kq_{n-1}}, \underline{0}) |};\; k=0,1,\cdots a_{n-1}.
\end{equation*}

\subsection{ Proof of the first part of  main result}
\subsubsection{A priori Bounds of $\alpha_n$ }
\begin{prop}\label{asymptotically}
	Let $n\in\N$ and $(\ell_1,\ell_2)\in\Omega_0=[1,2]^2$. 
	\begin{description}
		\item For all $\alpha_n$, 
		\begin{equation*}
		\alpha_{n}^\frac{\ell_1,\ell_2}{2}<0.55;
		\end{equation*}
		\item for at least every other $\alpha_n$
		\begin{equation*}
		\alpha_{n}^\frac{\ell_1,\ell_2}{2}<0.3.
		\end{equation*}
		\item	If 		\begin{equation*}\alpha_{n}^\frac{\ell_1,\ell_2}{2}>0.3 ,
		\end{equation*}	
		then either,
		\begin{equation*}
		\alpha_{n}^\frac{\ell_1,\ell_2}{2} <0.44 \quad \mbox{or}\quad \alpha_{n+1}^\frac{\ell_1,\ell_2}{2}<0.16.
		\end{equation*}	
		
	\end{description}	
\end{prop}

\begin{proof}
	For every $n\in\N$ and $k=0,1,\cdots a_{n-1} $, we define the parameter sequences 
	\begin{equation*}
	\gamma_{1,n}(k):=|(\underline{-q_{n}+kq_{n-1}}, \underline{0})|,\quad \gamma_{1,n-1}:=\gamma_{1,n-1}(0),\quad \gamma_n(k) = \dfrac{\gamma_{1,n}(k)}{\gamma_{1,n-1}}
	\end{equation*}	
	\begin{equation*}\gamma_n^{(\ell_1|\ell_2)}(k):=
	\dfrac{\gamma^{\eg}_{1,n}(k)}{\gamma^{\ed}_{1,n-1}} \hspace{0.9mm} \mbox{  if } \hspace{0.9mm} n\in 2\Z \hspace{1mm}\mbox{ and }\hspace{1mm}\gamma_n^{(\ell_1|\ell_2)}(k):=\dfrac{\gamma^{\ed}_{1,n}(k)}{\gamma^{\eg}_{1,n-1}}
	\hspace{0.9mm}\mbox{ if } \hspace{0.9mm} n\in 2\Z +1.
	\end{equation*}		
	These notations simplify the  formalization  of the following lemma which will play an important (essential) role in the proof of \textbf{Proposition~\ref{asymptotically}}.
	
	\begin{lem}\label{use additional condition}
		For $n$  large enough and  for every		$k=0,1,\cdots a_{n-1}-1 $, the following inequality holds  
		\begin{equation}\label{lemma inequality 1}
		\dfrac{(\beta_n(k)^{\ell_1,\ell_2}+ \alpha^{\ell_1,\ell_2}_{n-1}\gamma_n^{(\ell_1|\ell_2)}(k)  )(1+\gamma_n^{(\ell_1|\ell_2)}(k)  )}{(1+ \alpha^{\ell_1,\ell_2}_{n-1}\gamma_n^{(\ell_1|\ell_2)}(k))(\beta_n(k)^{\ell_1,\ell_2}+\gamma_n^{(\ell_1|\ell_2)}(k) ) }\leq s_n\beta_n(k+1).
		\end{equation}

	\end{lem}
	
	\begin{proof} Let  $n$ be an even non negative integer large enough.
		For fixed 	$k=0,1,\cdots a_{n-1}-1,$ according to the assumption \textbf{(\ref{add assump 2})}, the left hand side of 	\textbf{(\ref{lemma inequality 1})} is equal to the cross-ratio 		
		\begin{equation*}
		\po\,(\underline{-q_{n}+kq_{n-1}+1}, \underline{-q_{n-1}+1}).  
		\end{equation*}
		Applying $f^{q_{n-1}-1} $, by expanding cross-ratio property, we get the inequality.	
	\end{proof}

	The left hand side  is a function of the three	variables $\beta_n(k)^{\ell_1,\ell_2} $, $\alpha^{\ell_1,\ell_2}_{n-1} $ and $\gamma_n^{(\ell_1|\ell_2)}.$ 
	Observe that the function increases monotonically with each of the first two variables. However, relatively to the third variable, the function reaches a minimum. To see this, take the logarithm of the function and
	check that the first derivative is equal to zero only when	
	\begin{equation*}
	(\gamma_n^{(\ell_1|\ell_2)})^2=\dfrac{\beta_n(k)^{\ell_1,\ell_2}}{\alpha^{\ell_1,\ell_2}_{n-1}}.
	\end{equation*}	
	By substituting this value for $\gamma_n^{(\ell_1|\ell_2)}$ in \textbf{(\ref{lemma inequality 1})}, we get that
	\begin{equation}\label{quadradic inquality}
	\left(\dfrac{\beta_n(k)^{\frac{\ell_1,\ell_2}{2}}+ \alpha^{\frac{\ell_1,\ell_2}{2}}_{n-1}}{ 1+\beta_n(k)^{\frac{\ell_1,\ell_2}{2}} \alpha^{\frac{\ell_1,\ell_2}{2}}_{n-1}}\right)^2\leq s_n\beta_n(k+1).
	\end{equation}
	Put:
	\begin{equation*}x_{n}(k):=\min\{\beta_n(k)^{\frac{\ell_1,\ell_2}{2}}, \alpha^{\frac{\ell_1,\ell_2}{2}}_{n-1}    \} \quad\mbox{ and } \quad y_{n}(k):=\beta_n(k)^{\frac{\ell_1,\ell_2}{2}}	.	\end{equation*}
	
	Since $\beta_n(k+1)\leq y_{n}(k+1) $, substituting the above variable into
	\textbf{(\ref{quadradic inquality})} gives rise to a quadratic inequality in $x_{n}(k) $ whose only root in the interval $(0,1)$ is given by
	\begin{equation*} \dfrac{ \sqrt{s_n  y_{n}(k+1)} }{1+\sqrt{1-s_n  y_{n}(k+1)}};	\end{equation*}
	that is,
	\begin{equation}\label{inequation min} x_{n}(k)=\min\{\beta_n(k)^{\frac{\ell_1,\ell_2}{2}}, 
	\alpha^{\frac{\ell_1,\ell_2}{2}}_{n-1}\}\leq\dfrac{\sqrt{s_n y_{n}(k+1)}}{1+\sqrt{1-s_n y_{n}(k+1)}}.	\end{equation}

	
	\begin{lem}\label{subsequence}
		There is a subsequence of $\alpha_n$ including at least every other $\alpha_n$,
		such that 	
		\begin{equation*}	\limsup \alpha_{n}^{\frac{\ell_1,\ell_2}{2}}\leq 0.3.	\end{equation*}	
	\end{lem}
	\begin{proof}
		We use the following elementary lemma in \cite{5}.
		\begin{lem} \label{function h}
			The function 
			\begin{equation*} 
			h_n(z)=\dfrac{ \sqrt{s_n z }}{1+\sqrt{1-s_n z }}		
			\end{equation*}
			moves points to the left, $h(z)<z$, if $z\geq 0.3$ and $n$ is large enough. 
		\end{lem}
		We select the subsequence.
		\begin{enumerate}
			\item \textbf{The initial term:}
			there exists $n-2\in\N$, such that $\alpha_{n-2}^{\frac{\ell_1,\ell_2}{2}}\leq 0.3$. This comes directly from the properties of the function $h_n$ (\textbf{Lemma~\ref{function h}}) and from \textbf{(\ref{inequation min})}. 
			\item \textbf{The next element:} suppose  that $\alpha_{n-2} $ has been selected. If 
			\begin{equation*}
			x_{n}(k)=\alpha_{n-1}^{\frac{\ell_1,\ell_2}{2}}\quad\mbox{or }\quad\alpha_{n-1}^{\frac{\ell_1,\ell_2}{2}} \leq 0.3,	\end{equation*}
			for some $k=0,1,\cdots a_{n-1}-1, $
			then, we select $\alpha_{n-1} $ as the next term. Otherwise, $\alpha_{n} $
			is the next term. Thus, the sequence is constructed.
		\end{enumerate}
	\end{proof}
	
	\begin{cor}
		For the whole sequence $(\alpha_n)$ we have 
		\begin{equation*}	\limsup \alpha_{n}^{\frac{\ell_1,\ell_2}{2}}\leq 0.3	\end{equation*}	
		Moreover, 
		\begin{description}
			\item[-] if $\alpha_{n-1} $ does not belong to the subsequence $(\alpha_n)$  defined by the 
			\textbf{Lemma~\ref{subsequence}} then either
			\begin{equation*}	\alpha_{n-1}^{\frac{\ell_1,\ell_2}{2}}<0.44\quad \mbox{or}\quad 	\alpha_{n}^{\frac{\ell_1,\ell_2}{2}}<0.16 	\end{equation*}							
		\end{description}
	\end{cor}
	\begin{proof}
		Observe that the function 
		\begin{equation*}
		H:(s,t)\in\R_+^2\mapsto F(s,t)=\dfrac{s+t}{1+st}
		\end{equation*}
		is symmetric and for fixed $s$, the function $F(s,\cdot)$
		reaches its minimum in zero by taking the value $s$. Therefore, for every $s,t\geq 0$,
		\begin{equation*}
		s,t\leq\dfrac{s+t}{1+st}. 
		\end{equation*}
		So,
		\begin{equation} \label{inquality alpha_n alpha_n-1}
		\alpha_{n}^{\frac{\ell_1,\ell_2}{2}}, \;\alpha_{n-1}^{\frac{\ell_1,\ell_2}{2}}\leq \dfrac{\alpha_{n}^{\frac{\ell_1,\ell_2}{2}}+ \alpha_{n-1}^{\frac{\ell_1,\ell_2}{2}}}{ 1+\alpha_{n}^{\frac{\ell_1,\ell_2}{2}} \alpha_{n-1}^{\frac{\ell_1,\ell_2}{2}}}.
		\end{equation}
		Thus, according to that $\alpha_{n-2} $ is an element of the sequence and  suppose that $\alpha_{n-1} $ do not belong to the previous subsequence of the \textbf{Lemma~\ref{subsequence}}, then,  it follows from \textbf{(\ref{quadradic inquality})}  that the right member of \textbf{(\ref{inquality alpha_n alpha_n-1})}  is estimated as following
		\begin{equation}\label{estimation ratio of quadratic}
		\dfrac{\alpha_{n}^{\frac{\ell_1,\ell_2}{2}}+ \alpha_{n-1}^{\frac{\ell_1,\ell_2}{2}}}{ 1+\alpha_{n}^{\frac{\ell_1,\ell_2}{2}} \alpha_{n-1}^{\frac{\ell_1,\ell_2}{2}}}\leq\sqrt{s_n\beta_n(1)^{\frac{\ell_1,\ell_2}{2}}}\approx \sqrt{\beta_n(1)^{\frac{\ell_1,\ell_2}{2}}} \leq\sqrt{ 0.3}.
		\end{equation}
		Also,
		\begin{equation}\label{alphan is min}
		\min\{\alpha_{n}^{\frac{\ell_1,\ell_2}{2}}, \alpha_{n-1}^{\frac{\ell_1,\ell_2}{2}}   \}=\alpha_{n}^{\frac{\ell_1,\ell_2}{2}}\leq 0.3.	\end{equation}
		Thus, if   $\alpha_{n}^{\frac{\ell_1,\ell_2}{2}}\geq 0.16,  $
		then by combining this with \textbf{(\ref{estimation ratio of quadratic})} and \textbf{(\ref{alphan is min})}, we obtain the desired estimate.   \end{proof}
	The \textbf{Proposition~\ref{asymptotically}} is proved.
\end{proof}

\subsubsection{Recursive formula of  $\alpha_n$}

\begin{prop}\label{recurrence formula}
	Let $n $ be integer large enough,
	\begin{enumerate}
		\item if $\ell_{1},\,\ell_{2}>1$,  we have  
		\begin{equation} \label{inequation recurrence formula}
		\alpha_{2n}^{\ell_{1}}\leq M_{2n}(\ell_{1})\alpha^{2}_{2n-2}\quad\mbox{and} \quad
		\alpha_{2n+1}^{\ell_{2}}\leq M_{2n+1}(\ell_{2})\alpha^{2}_{2n-1};
		\end{equation}		
		where,\begin{equation*}
		M_n(\ell)=s^2_{n-1}\cdot \dfrac{2}{\ell}\cdot \dfrac{1}{1+\sqrt{1-\dfrac{2(\ell-1)}{\ell}s_{n-1}\alpha_{n-1}}}
		\cdot \dfrac{1}{1-\alpha_{n-2}}\cdot\frac{\sigma_{n}}{\sigma_{n-2}} .
		\end{equation*}
		\item if $\ell_{1}=\ell_{2}=1$, then
		\begin{equation} \label{inequation recurrence formula1}
		\alpha_{n}\leq W^1_{n}\cdot\frac{\sigma_{n}}{\sigma_{n-2}}\alpha_{n-2}.
		\end{equation}		
	\end{enumerate}	
\end{prop}	
\begin{proof} We treat the case $n$ even and the case $n$ odd is treated in a similar way. Recall that, 
	\begin{equation*}
	\alpha_n=\dfrac{|(\underline{-q_n}, \underline{0})|}{|[\underline{-q_n}
		, \underline{0})|}.
	\end{equation*}
	For every $n$ even large enough, by  \textbf{Proposition~\ref{qn go to zero}} and point 2 of  \textbf{Fact~\ref{fact}}, applying $f$ to the equality, we get
	\begin{equation*}
	\alpha^{\ell_{1}}_n=\dfrac{|(\underline{-q_n+1}, \underline{1})|}{|[\underline{-q_n+1}
		, \underline{1})|}.
	\end{equation*}
	which is certainly less than the cross-ratio
	\begin{equation*}
	\po(\underline{-q_n+1}, (\underline{1},\underline{-q_{n-1}+1}]).
	\end{equation*}
	Since the cross-ratio \po\, is expanded by $f^{q_{n-1}-1} $, then,
	\begin{equation}\label{rn4}
	\alpha^{\ell_{1}}_n< \delta_n(1)s_n(1);
	\end{equation}
	with,
	\begin{equation*}
	\delta_n(k):=\dfrac{|(\underline{-q_{n}+kq_{n-1}}, \underline{kq_{n-1}})|}{|[\underline{-q_{n}+kq_{n-1}}
		, \underline{kq_{n-1}})|}.
	\end{equation*}
	and 
	\begin{equation*}
	s_n(k):=\dfrac{|[\underline{-q_{n}+kq_{n-1}}, \underline{0}]|}{|(\underline{-q_{n}+kq_{n-1}}
		, \underline{0}]|}.
	\end{equation*}
	
	If $a_{n-1}=1$, 
	by multiplying and dividing the right member of \textbf{(\ref{rn4})} by 
	$\alpha^2_{n-2} ,$ we obtain directly \textbf{(\ref{rn5})}. 
	
	Suppose that $a_{n-1}>1$ and 
	estimate $\delta_n(k)$. By the Mean Value Theorem (Lagrange), $f$ transforms the intervals
	defining the ratio $\delta_n(k)$ into a pair whose ratio is
	\begin{equation*}\dfrac{u_k}{v_k}\delta_n(k)\end{equation*}
	with $u_k$ being the derivative of $f\,(x^{\ell_{1}})$ at a point in the interval
	\begin{equation*}
	U_k:=(\underline{-q_{n}+kq_{n-1}}, \underline{kq_{n-1}}),
	\end{equation*}
	and
	$v_k$ being the derivative of $f\,(x^{\ell_{1}})$ at a point in the interval
	\begin{equation*}
	V_k:=[\underline{-q_{n}+kq_{n-1}}
	, \underline{kq_{n-1}}).
	\end{equation*}
	Note that, for $n$ sufficiently large, 
	\begin{equation*}
	u_1<v_1<u_2<v_2\cdots <v_{a_{n-1}}.
	\end{equation*}
	
	We see that the image of $\delta_n(k)$  by $f$ is smaller than:
	\begin{equation*}
	\po(\underline{-q_n+kq_{n-1}+1}, (\underline{kq_{n-1}+1},\underline{-q_{n-1}+1}]).
	\end{equation*}
	Once again, by  expanding cross-ratio property $(f^{q_{n-1}-1}) $, it follows that,
	\begin{equation}\label{delta_n s_n k}\dfrac{u_k}{v_k}\delta_n(k)\leq s_n(k+1)\cdot\delta_n(k+1). \end{equation}
	Multiplying \textbf{(\ref{delta_n s_n k})} for $k=0,..., a_{n-1}-1$, and substituting the resulting estimate of $\delta_n(1)$ into  \textbf{(\ref{rn4}) }, we obtain:
	\begin{equation*}\alpha^{\ell_{1}}_n\leq \delta_n(a_{n-1})\dfrac{v_{a_{n-1}}}{u_1} s_n(1)\cdots s_n(a_{n-1}). \end{equation*}
	Observe that, $s_n(1)\cdots s_n(a_{n-1})<s_n$ and 
	\begin{equation*}\dfrac{v_{a_{n-1}}}{u_1}\leq\left(\dfrac{|(\underline{-q_{n-2}},\underline{0})|}{|(\underline{q_{n-1}}, \underline{0})|}\right)^{\ell_1-1}
	\leq \dfrac{|(\underline{-q_{n-2}},\underline{0})|}{|(\underline{q_{n-1}}, \underline{0})|}.
	\end{equation*}
	Thus,
	\begin{equation*}\alpha^{\ell_{1}}_n\leq s_n\cdot \dfrac{|(\underline{-q_{n-2}},\underline{0})|}{|(\underline{q_{n-1}}, \underline{0})|}\cdot \delta_n(a_{n-1});
	\end{equation*}
	which can be rewritten in the form
	\begin{equation}
	\alpha^{\ell_{1}}_n\leq s_n\nu_{n-2} \mu_{n-2} \alpha^2_{n-2};\label{rn5}
	\end{equation}
	with,
	\begin{equation*}
	\nu_{n-2}:=\dfrac{|[\underline{-q_{n-2}}, \underline{0})|}{|(\underline{q_{n-1}}
		, \underline{0})|} \cdot \dfrac{|[\underline{-q_{n-2}}, \underline{0})|}{|[\underline{-q_{n-2}}
		, \underline{a_{n-1}q_{n-1}})|}
	\end{equation*}
	and
	\begin{equation*}
	\mu_{n-2} :=\dfrac{|(\underline{-q_{n-2}}, \underline{a_{n-1}q_{n-1}})|}{|(\underline{-q_{n-2}}
		, \underline{0})|}.
	\end{equation*}

	It remains to estimate $\nu_{n-2}$ and $\mu_{n-2}$ to end this part.
	For $\nu_{n-2}$, observe that,
	\begin{equation*}
	|(\underline{-q_{n-2}}, \underline{0})|\leq |(\underline{q_{n-3}}, \underline{0})|
	\end{equation*}
	so that,
	\begin{equation}\label{rn6}
	\nu_{n-2}\leq\dfrac{    1  }{\sigma_{n-1} \sigma_{n-2}}\cdot \dfrac{1}{1-\alpha_{n-2}}
	\end{equation}	
	The estimation of $\mu_{n-2}$ is facilitated by the following elementary lemma in \cite{5}.
	\begin{lem}\label{rn7}
		Let $\ell\in(1,2)$. For all numbers $x>y$, we have the following inequality:
		\begin{equation*}
		\dfrac{x^\ell-y^\ell}{x^\ell}\geq \left( \dfrac{x-y}{x}\right)\left[\ell- \dfrac{\ell(\ell-1)}{2}\left(\dfrac{x-y}{x}\right)\right].
		\end{equation*}
		
	\end{lem}
	
	Now, apply $f$ into the intervals defining the ratio $\mu_{n-2} $. By \textbf{Lemma~\ref{rn7}}, the resulting ratio is larger than

	\begin{equation*}
	\mu_{n-2}( \ell_{1}-\dfrac{\ell_{1}(\ell_{1}-1)}{2} \mu_{n-2}).
	\end{equation*}
	The cross-ratio \po
	\begin{equation*}
	\po\:(-q_{n-2}+1, (\underline{q_{n-1}+1}, \underline{1}));
	\end{equation*}
	that is,
	\begin{equation*}
	\dfrac{|(\underline{-q_{n-2}+1},  \underline{q_{n-1}+1})| |[\underline{-q_{n-2}+1}, \underline{1})|}
	{|[\underline{-q_{n-2}+1},  \underline{q_{n-1}+1})| |(\underline{-q_{n-2}+1}, \underline{1})|}
	\end{equation*}
	is larger again. Thus, by expanding cross-ratio  property on $f^{q_{n-2}} $, we obtain:
	\begin{equation}\label{rn8}
	\mu_{n-2}( \ell_{1}-\dfrac{\ell_{1}(\ell_{1}-1)}{2} \mu_{n-2})\leq s_{n-1}\sigma_{n}\sigma_{n-1}.
	\end{equation}
	By solving this quadratic inequality, we obtain
	\begin{equation}\label{rn9}
	\mu_{n-2}\geq \dfrac{1+\sqrt{1-\dfrac{2}{\ell_{1}}(\ell_{1}-1) s_{n-1}\sigma_{n}\sigma_{n-1}}}{\ell_{1}-1}.
	\end{equation}
	Thus, by combining the  \textbf{(\ref{rn8})} and \textbf{(\ref{rn9})}, we obtain
	\begin{equation}\label{rn11}
	\mu_{n-2}<\dfrac{2}{\ell_{1}}\cdot \dfrac{1}{1+\sqrt{1-\dfrac{2(\ell_{1}-1)}{\ell_{1}} s_{n-1}\sigma_{n}\sigma_{n-1}}}
	s_{n-1}\sigma_{n}\sigma_{n-1}.
	\end{equation}
	Since, $\sigma_{n}\sigma_{n-1}<\alpha_{n-1} $, the first inequality in \textbf{(\ref{inequation recurrence formula})}  follows  by combining the inequalities
	\textbf{(\ref{rn5})}, \textbf{(\ref{rn6})} and \textbf{(\ref{rn11})}. Likewise, the second inequality in \textbf{(\ref{inequation recurrence formula})}
	is obtained by following suitably the same reasoning as previously. 
	
\end{proof}

\subsubsection{  $\alpha_n$ go to zero}
If $\ell_{1}=\ell_{2}=1 $, then by \textbf{Proposition~\ref{asymptotically}}, $\prod_{k=2}^{k=n}W^1_{k} $ goes to zero; thus, by composing the inequality obtained 
by $f$, since $f(\sigma_n)$ is bounded ( \textbf{Lemma~\ref{sigma go a away}} ), then the result follows.

Note that, the cases where  the critical exponents are of the form $(1, \ell)$ or $(\ell, 1)$; with $\ell>1$, are treated  in \cite{4}.

Now,  Let us suppose that  $\ell_{1},\ell_{2}>1$.

Technical reformulation of the \textbf{Proposition~\ref{recurrence formula}}.

Let  $W_n$ be a sequence  defined by
\begin{equation*}
M_n(\ell)=W_n(\ell)\dfrac{\sigma_{n}}{\sigma_{n-2}}.
\end{equation*}
Let 
\begin{equation*}
M'_n(\ell): =M_n(\ell)\alpha_{n-2}^{2-\ell} \quad\mbox{and}\quad W'_n(\ell): =W_n(\ell)\alpha_{n-2}^{2-\ell}
\end{equation*}
The recursive formula \textbf{(\ref{inequation recurrence formula})} can be written for $n$ even  in the form:
\begin{equation*}
\alpha_{n}^{\ell_1} \leq W'_n(\ell_1)\dfrac{\sigma_{n}}{\sigma_{n-2}}\alpha_{n-2}^{\ell_1}.
\end{equation*}
so, 
\begin{equation*}
\alpha_{n}^{\ell_1} \leq \prod_{k=2}^{k=n} W'_k(\ell_1)\dfrac{\sigma_{n}}{\sigma_{0}}\alpha_{0}^{\ell_1}.
\end{equation*}

$\prod_{k=2}^{k=n} W'_k(\ell_1)$ goes to zero

Observe that, the size of $	W'_n(\eg)$ is given by the study of the function	
\begin{equation*}
W'_n(x,y,\eg)= \dfrac{1}{	\dfrac{\eg}{2}+	\dfrac{\eg}{2}\sqrt{1-\dfrac{2(\eg-1)}{\eg}x^{\frac{2}{\ed}}}}
\cdot \frac{y^{\frac{4}{\eg}-2}}{1-y^{\frac{2}{\eg}}}
\end{equation*}
The meaning of variation of  $W'_n(x,y,\eg)$ relative to the third variable is given by the following lemma in \cite{5}).
\begin{lem}\label{lemma for conclusion}
	For any $0<y<\frac{1}{\sqrt{e}}$, $x\in(0,1)$ and $\eg\in(1,2]$  the function $	W'_n(x,y,\eg)$ is increasing with respect to $\eg$.
\end{lem}

\subsubsection*{Analyse the asymptotic size of $W'_i(2)$ }
Since the hypotheses of the \textbf{Lemma~\ref{lemma for conclusion}} are satisfied (\textbf{Proposition~\ref{asymptotically}}),
the only remaining point is the verification of the convergence of $\prod_{i=1}^{n}W'_i(2)$. 

\begin{description}
	\item[\textbf{-}]If $\alpha_{ n-2} < (0.3)^{\eg}$, then 
	$W'(2)<W'(0.55,0.16,2)<0,9$.
	\item[\textbf{-}] If not, then by the \textbf{Proposition~\ref{asymptotically}}, 	$W'(2)<W'(0.3,0.44,2)<0,98$
	or else,  $W_{n+1}'(2)W_{n}'(2)<W'(0.55,0.16,2)W'(0.16,0.55,2)<0,85$
\end{description}
\begin{cor}\label{alpha_n go to zero double exponentially}
	Let $\ell_1,\ell_2\in [1,2]$. If $1<\ell_1<2$ respectively $1<\ell_2<2$, then
	$\alpha_{2n} $ respectively 	$\alpha_{2n+1} $ goes to zero least double exponentially fast. And, if $\ell_1=2$ or $1$  respectively $\ell_2=2$ or $1$, then
	$\alpha_{2n} $ respectively 	$\alpha_{2n+1} $ goes to zero least  exponentially fast
\end{cor}
\begin{proof} Let $n:=2p_n\in\N$.
	From the analysis of the asymptotic size of $W'_i(2)$, it follows that, when $n$ goes to infinity, $\prod_{i=0}^{n}W'_{i}(\ell_1)$ goes to zero and $\alpha_n $ does so. Therefore, 
	
	\begin{equation*}
	\prod_{i=0}^{n}M_i(\ell_1)
	\end{equation*}	
	goes to zero when $n$ goes to infinity. Thus, by the \textbf{Proposition~\ref{recurrence formula}}, for $n$ even, there is $\lambda_0$ such that
	\begin{enumerate}
		\item[-]if $1<\ell_1<2$,  
		\begin{equation*}\alpha_{n}\leq\lambda _0^{\left(\frac{2}{\ell_1}\right)^{p_{n}}}, \end{equation*}	 
		\item[-] and if $\ell_1=1,2$,
		\begin{equation*}\alpha_{n}\leq\lambda_0^{n}. \end{equation*}	
	\end{enumerate}
	The case $n$ odd is treated the same way.
\end{proof}

\subsection{ Proof of the second part of main result}

In this section we find a bounded geometry domain.

\subsubsection{Recursive Affine Inequality  of order two on $\alpha_n$}
Let 
\begin{equation*}
\kappa_n:=\dfrac{|(\underline{0}, \underline{q_{n}}) |}{|(\underline{0}, \underline{-q_{n-1}}) |}.
\end{equation*}	

\begin{rem}\label{rem tau alpha kappa}
	Since the point $\underline{q_{n-2}} $ lies in the gap between $\underline{-q_{n-1}} $ and\\ $\underline{-q_{n-1}+q_{n-2}} $ of the dynamical partition $\mathcal{P}_{n-2} $, then by the\\ \textbf{Proposition~\ref{preimage and gaps adjacent}}, $\tau_n/{\alpha_{n-1}}$ and $\kappa_n$ are comparable.	
\end{rem}

\begin{prop}{\label{kappa go away to 0}} For any bounded type rotation number, there is a uniform constant $K$ so that
	\begin{equation*}
	\kappa_{2n}>K\left(\alpha_{2n-1}\right)^{\dfrac{1-\ell_2^{a_{2n}+1}}{\ell_2-1}}\quad\mbox{and}\quad \kappa_{2n+1}>K\left(\alpha_{2n}\right)^{\dfrac{1-\ell_1^{a_{2n+1}+1}}{\ell_1-1}}. \end{equation*}
\end{prop}

\paragraph{Proof of the Proposition:} 
If $a_n=1$, it comes down to showing that the sequence $\kappa_n$ is bounded away from zero; which becomes relatively very simple. In fact, 
suppose that, 
$|( \underline{q_n},\underline{-q_{n-1}})|\leq |(\underline{0}, \underline{q_{n}})|,$
then 
\begin{equation*}
\kappa_n=\dfrac{|(\underline{0}, \underline{q_{n}}) |}{|(\underline{0}, \underline{-q_{n-1}}) |}\geq \dfrac{1}{2},
\end{equation*}
else,   $\kappa_n$ is greater than
\begin{equation*}
\dfrac{ | \underline{-q_{n+1}}|}{|[ \underline{-q_{n+1}},\underline{-q_{n-1}})|};
\end{equation*}
which by the \textbf{Proposition~\ref{preimage and gaps adjacent}} is bounded away from zero. 

In the following part of the proof, we suppose that $a_n>1$ and the following lemma in \cite{5} (\textbf{Lemma 4.1.}) is used.
\begin{lem} \label{1-beta_n}
	The ratio
	\begin{equation*}
	1-\beta_n(i)=\dfrac{|\underline{-q_{n}+iq_{n-1}} |}{|[\underline{-q_{n}+iq_{n-1}}), \underline{0}) |} 	
	\end{equation*}
	is bounded away from zero by a uniform constant for all $i=0,\cdots, a_{n-1}. $
\end{lem}

\begin{lem}\label{beta_n power ell}
	For every $n\in\N$ and for all $i=0,\cdots, a_{n-1}	$, there is a uniform constant $K$ such that 
	\begin{equation*}
	\beta_n(i)^{\ell_1,\ell_2 }\geq \beta_n(i+1).	
	\end{equation*}
\end{lem}

\begin{proof}
	For $n$ large enough and for fixed $i=0,\cdots, a_{n-1}	$, by the \textbf{Proposition~\ref{qn go to zero}} and the 
	\textbf{Fact~\ref{fact}}, we have:
	\begin{equation*}
	\beta_n(i)^{\ell_1,\ell_2 }=\dfrac{|(\underline{-q_{n}+kq_{n-1}+1}, \underline{1})|}{|[\underline{-q_{n}+iq_{n-1}+1}, \underline{1}) |};
	\end{equation*}
	which is greater than,
	\begin{equation*}
	\cro([\underline{-q_{n-2}+1}, \underline{-q_{n}+iq_{n-1}+1}),(\underline{-q_{n}+iq_{n-1}+1},\underline{1} )).	\end{equation*}	
	By applying the cross-ratio inequality under $f^{q_{n-2}-1} $, by the \textbf{Fact~\ref{fact}} the resulting ratio    is greater than	
	\begin{equation*}\begin{array}{r}
	\cro([\underline{-q_{n-2}+1}, \underline{-q_{n}+iq_{n-1}+q_{n-2}+1}),\vspace{0.1cm}\\(\underline{-q_{n}+iq_{n-1}+q_{n-2}+1},q_{n-2}+\underline{1} ))\end{array}	
	\end{equation*}		
	times a uniform constant.
	We now repeat this sequence of steps $a_{n-2}-1$ times: Apply $f^{q_{n-2}-1} $, discard the interval containing $\underline{0} $ and use  the point \textbf{2} of the \textbf{Fact~\ref{fact}}, and replace  the result by a cross-ratio spanning the intervals $\underline{-q_{n-2}+1} $. At the end, this will produce the cross ratio
	\begin{equation*}\begin{array}{r}
	\cro([\underline{-q_{n-2}+1}, \underline{-q_{n}+iq_{n-1}+a_{n-2}q_{n-2}+1}),\vspace{0.1cm}\\(\underline{-q_{n}+iq_{n-1}+a_{n-2}q_{n-2}+1},a_{n-2}q_{n-2}+\underline{1} )).\end{array}	
	\end{equation*}
	And finally, by applying $f^{q_{n-3}-1} $, since by the \textbf{Lemma~\ref{1-beta_n}} the interval containing $\underline{-q_{n-2}+q_{n-3}} $ bounded away from zero, then resulting ratio is 
	\begin{equation*}
	\dfrac{|(\underline{-q_{n}+(i+1)q_{n-1}},\underline{q_{n-1}} ) |}{|[\underline{-q_{n}+(i+1)q_{n-1}},\underline{q_{n-1}} )|}
	\end{equation*}
	times a uniform constant. Thus, since $\underline{-q_{n}+q_{n-1}} $ lies between\\
	$\underline{-q_{n}+(i+1)q_{n-1}} $ and  $\underline{q_{n-1}} $, then, by the \textbf{Lemma~\ref{1-beta_n}}, this ratio is comparable to $\beta_{n}(i+1).$
\end{proof}
\paragraph{ Back to the proof of the \textbf{Proposition~\ref{kappa go away to 0}}:}
Observe by the \textbf{Proposition~\ref{preimage and gaps adjacent}} that 
\begin{equation*}
|(\underline{-q_{n}+(i+1)q_{n-1}},\underline{0} ) |
\end{equation*}
and 
\begin{equation*}
{|[\underline{-q_{n}+iq_{n-1}},\underline{0} )|}
\end{equation*}
are comparable. Therefore, $\kappa_{n-1} $ is comparable to the product
\begin{equation*}
\beta_n(1) \cdots\beta_n(a_{n-1}-1).
\end{equation*}
By combining this with the \textbf{Lemma~\ref{beta_n power ell}}, we have the \textbf{Proposition~\ref{kappa go away to 0}}.

\paragraph{Recursive Affine Inequality  of order two on $\alpha_n$}

\begin{prop}\label{alpha n, n-1 n-2} If $\rho(f)$ is of bounded type, then there is 	a uniform constant $K$ so that,  
	\begin{equation*}
	\alpha_{2n}\geq K\left(\alpha_{2n-1}\right)^{\dfrac{\ell_2}{\ell_1}\cdot\dfrac{1-\ell_2^{-a_{2n}}}{\ell_2-1}}\left(\alpha_{2n-2}\right)^{\ell_1^{-a_{2n-1}}}
	\end{equation*} and
	\begin{equation*}
	\alpha_{2n+1}\geq K\left(\alpha_{2n}\right)^{\dfrac{\ell_1}{\ell_2}\cdot\dfrac{1-\ell_1^{-a_{2n+1}}}{\ell_1-1}}\left(\alpha_{2n-1}\right)^{\ell_2^{-a_{2n}}}. 	\end{equation*}
\end{prop}
\paragraph{Proof of the Proposition}

If $n$ is even and large enough, then 
\begin{equation*}
\alpha^{\eg}_{n}=	\dfrac{ |( \underline{-q_{n}+1},\underline{1})|}{|[ \underline{-q_{n}+1},\underline{1})|};
\end{equation*}
which in turn is larger than the product of two ratios 
\begin{equation*}
\xi_{1,n}=	\dfrac{ |( \underline{-q_{n}+1},\underline{1})|}{|( \underline{-q_{n}+1},\underline{-q_{n-1}+1})|}
\quad \mbox{and} \quad 	\xi_{2,n}=	\dfrac{ |( \underline{-q_{n}+1},\underline{-q_{n-1}+1})|}{|[ \underline{-q_{n}+1},\underline{-q_{n-1}+1})|}
\end{equation*}

\begin{lem}\label{xi 1} For all $n$ even large enough
	\begin{equation*}
	\xi_{1,n}\geq K\tau_n.	
	\end{equation*}
\end{lem}
\begin{proof}
	Observe that $\xi_{1,n}$ is greater than 
	\begin{equation*}
	\cro\,(( \underline{-q_{n}+1},\underline{1}), \underline{-q_{n-1}+1} ).
	\end{equation*}
	By applying \cri\, on $f^{q_{n-1}-1}$ and discarding the intervals containing $\underline{0}$. Repeat this  $a_{n-1}-1$ times more: By the \textbf{Fact~\ref{fact}}, the resulting ratio  is large than
	\begin{equation*}
	\cro\,(( \underline{-q_{n-2}-q_{n-1}+1},\underline{(a_{n-1}-1)q_{n-1}+1}), (\underline{1}, \underline{-q_{n-1}+1} ])).
	\end{equation*}
	Apply 	$f^{q_{n-1}-1}$, and discard the intervals containing the flat interval. Apply $f$, replace the resulting by the cross-ratio 
	\begin{equation*}
	\cro\,(( \underline{-q_{n-2}+1},\underline{q_{n-1}+1}), (\underline{1}, \underline{-q_{n-1}+1} ])).
	\end{equation*}		
	Thus, by \cri\, on $f^{q_{n-2}-1}$	and the inequalities above, we obtain	
	\begin{equation}\label{e23}
	\xi_{1,n}>	\dfrac{ |( \underline{q_{n-2}},\underline{-q_{n-3}}]|}{|( \underline{q_{n}},\underline{-q_{n-3}}]|}\tau_n
	\end{equation}
	The first factor on the right hand side of \textbf{\ref{e23}} is greater than
	\begin{equation*}
	\dfrac{ | \underline{-q_{n-3}}|}{|( \underline{0},\underline{-q_{n-3}}]|};
	\end{equation*}	
	which by the \textbf{Proposition~\ref{preimage and gaps adjacent}} goes away from zero. The Lemma is proved.
\end{proof}

\begin{lem}\label{xi 2} There is a uniform constant $K$ so that, for all $n$ 
	\begin{equation*}
	\xi_{2,2n}\geq K\left( \alpha_{2n-2}\right)^{\ell_1^{-a_{2n-1}+1}} \quad\mbox{and } \quad \xi_{2,2n+1}\geq K\left( \alpha_{2n-1}\right)^{\ell_1^{-a_{2n}+1}}.
	\end{equation*}
\end{lem}

\begin{proof}
	If $a_{n-1} =1$, then $\xi_{2,n}$ is greater than 
	\begin{equation*}
	\cro\,([ \underline{-q_{n-2}+1},\underline{-q_{n}+1}), (\underline{-q_{n}+1},\underline{-q_{n-1}+1}) ).
	\end{equation*}
	By applying \cri\, $({q_{n-2}-1})$,  the ratio resulting is greater than
	\begin{equation*}
	\cro\,([ \underline{-q_{n-2}+1},\underline{-q_{n-1}+1}), (\underline{-q_{n-1}+1}, \underline{-q_{n-3}+1} ))
	\end{equation*}		
	times a uniform constant. Thus, by applying \cri\, to this ratio with $f^{q_{n-3}-1}$, inequalities above and the  \textbf{Proposition~\ref{preimage and gaps adjacent}} the result follows. 
	
	Now, suppose that $a_{n-1} >1$ then $\xi_{2,n}$ is greater than 
	\begin{equation*}
	\cro\,([ \underline{-q_{n}+q_{n-1}+1},\underline{-q_{n}+1}), (\underline{-q_{n}+1},\underline{-q_{n-1}+1}) ).
	\end{equation*}
	By applying $f^{q_{n-1}-1}$ to this ratio, it follows from the \textbf{Lemma~\ref{1-beta_n}} that
	\begin{equation*}
	\xi_{2,n}\geq K'\beta_n(1).
	\end{equation*}
	And the lemma follows from this by using \textbf{Lemma~\ref{beta_n power ell}} modulo the fact that 
	\begin{equation*}
	\beta_n(a_{n-1})=\alpha_{n-2}.
	\end{equation*}
\end{proof}

Combining the  \textbf{Lemma~\ref{xi 1}},   the \textbf{Lemma~\ref{xi 2}}, 	the 
\textbf{Proposition~\ref{kappa go away to 0}} and the \textbf{Remark~\ref{rem tau alpha kappa}},	the result of the \textbf{Proposition~\ref{alpha n, n-1 n-2}} follows.

\begin{rem}
	By the inequality obtained from the \textbf{Proposition~\ref{alpha n, n-1 n-2}}, the  sequence $\nu_{n}$ defined by
	\begin{equation*}
	\nu_{n}=-\ln\alpha_n
	\end{equation*}	
	verifies the following 	Recursive Affine Inequalities of order two for every  $n>0$ 
	\begin{equation}\label{nu 1}
	\nu_{2n} \leq\dfrac{\ell_2}{\ell_1}\cdot t_2(a_{2n})\nu_{2n-1}+\ell_1^{-a_{2n-1}}\nu_{2n-2} +\widetilde{K'} 
	\end{equation}
	and 	
	\begin{equation}\label{nu 2}
	\nu_{2n+1} \leq \dfrac{\ell_1}{\ell_2}\cdot t_1(a_{2n+1})\nu_{2n}+\ell_2^{-a_{2n}}\nu_{2n-1} +\widetilde{K'};  
	\end{equation}
	with 
	\begin{equation*}
	t_i(j)=\dfrac{1-\ell_i^{-j}}{\ell_i-1}.
	\end{equation*}
\end{rem}	

\subsubsection{Analysis of Recursive Affine Inequality  }

We will prove that the sequence $\nu_{n}$ is bounded. Let us consider the sequence of vectors $(v_n)$ defined by
\begin{equation*}
v_n=\left(	\begin{array}{c}
\nu_{n}\\\nu_{n-1}
\end{array}\right),
\end{equation*}
the vector given by
\begin{equation*}
\overline{\kappa}	=\left(	\begin{array}{c}
\widetilde{K'}\\0
\end{array}\right)
\end{equation*}
and the sequence  matrix 
\begin{equation*}
A_{\eg,\ed}(2n)	=\left(	\begin{array}{cc}
\dfrac{\ell_2}{\ell_1}\cdot t_2(a_{2n}) & \ell_1^{-a_{2n-1}}\\1 & 0
\end{array}\right)\end{equation*} and
\begin{equation*} A_{\eg,\ed}(2n+1)	=\left(	\begin{array}{cc}
\dfrac{\ell_1}{\ell_2}\cdot t_1(a_{2n+1}) & \ell_2^{-a_{2n}}\\1 & 0
\end{array}\right);
\end{equation*}
say associated matrix  to the Recursive Affine Inequalities \textbf{(\ref{nu 1})} and \textbf{(\ref{nu 2})} respectively; in this sense that  \textbf{(\ref{nu 1})} and \textbf{(\ref{nu 2})} can be rewritten respectively in the form:
\begin{equation}
v_{2n}\leq	A_{\eg,\ed}(2n)v_{2n-1} + \overline{\kappa}\quad \mbox{and} \quad 
v_{2n+1}\leq A_{\ed,\eg}(2n+1)v_{2n}  + \overline{\kappa}.
\end{equation}
Therefore, for every $n:=2p_n+r_n$; with $p_n\in \N^*$ and  $r_n\in\{0,1\}$, we have
\begin{align*}
v_{n}\leq &	\overline{A}_{\eg,\ed}(n)\overline{A}_{\eg,\ed}(n-2)\cdots\overline{A}_{\eg,\ed}(2+r_n) v_{2-r_n} +\\ &  
\left(Id+
\sum_{i=4+r_n}^{n}\overline{A}_{\eg,\ed}(n)\overline{A}_{\eg,\ed}(n-2)\cdots\overline{A}_{\eg,\ed}(i)\right)\overline{\kappa'}.
\end{align*}
where,
\begin{equation*}
\overline{A}_{\eg,\ed}(n)=A_{\eg,\ed}(n)	A_{\eg,\ed}(n-1). 
\end{equation*}
Observe that, if $	(\eg,\ed)$ is very close to an element of the set   
\begin{equation*}
\{(a,\infty),\; (\infty,b),\; (\infty,\infty),\;a,b\in\R\}
\end{equation*}
$\overline{A}_{\eg,\ed}(n)$ is diagonalizable with nonnegative eigenvalues and at most one is strictly positive; that is,  ${1}/{\ell_1} $ or ${1}/{\ell_2} $ and 
as $\ell_1, \ell_2>2 $, then, $\overline{A}_{\eg,\ed}(n)$ contracts the Euclidean metric; therefore, $ v_n \mbox{ (also } \nu_n \mbox{ and } \alpha_n \mbox{)} $
is bounded. In the following analysis, we suppose that $ 1<\ell_1, \ell_2<C<\infty $; for some $C$ in $\R_+$.

\begin{lem}\label{relatively compact}
	Fix $(\eg,\ed)\in[2, \infty)^2$. The sequence 
	\begin{equation*}
	\overline{A}^{\circ n}_{\eg,\ed}:=\overline{A}_{\eg,\ed}(n)\overline{A}_{\eg,\ed}(n-2)\cdots\overline{A}_{\eg,\ed}(4)  
	\end{equation*}
	is bounded (uniformly) by $\max\{{\eg}/{\ed}, {\ed}/{\eg} \}.$
	
\end{lem}
\begin{proof}
	Observe that, for all $n\in\N$, 
	\begin{equation*}
	\overline{A}_{\eg,\ed}(n)\leq\left(	\begin{array}{cc}
	( 1-b_n(2))( 1-b_{n-1}(2)) & \frac{\ell_{1}}{\ell_{2}}( 1-b_n(2))b_{n-2}(2)\\\frac{\ell_{2}}{\ell_{1}}( 1-b_{n-1}(2)) & b_{n-2}(2)
	\end{array}\right)=:\overline{B}_{\eg,\ed}(n);
	\end{equation*}
	with, 
	\begin{equation*}
	b_{2n}(\ell)=b_{2n}:=\ell_2^{-a_{2n}} \quad\mbox{and}\quad b_{2n+1}(\ell)=b_{2n+1}:=\ell_1^{-a_{2n+1}}.
	\end{equation*}
	Let be a sequence $(x_{n})_{n\in\N} $   defined by:
	\begin{equation*}
	(x_{n})_{n\in\N}:=\{1- b_{n}(2),b_{n}(2); n\in\N\}.
	\end{equation*}
	Remark that for every $n\in\N$, $x_n\in[2^{-a},1-2^{-a}]$; where, 
	$a:=\max\{a_n, n\in\N\} $. 
	
	Thus, by setting 
	\begin{equation*}
	\overline{B}^{\circ n}_{\eg,\ed}=\left(	\begin{array}{cc}
	d^{1,n} & \frac{\ell_{2}}{\ell_{1}}d^{3,n}\\\frac{\ell_{1}}{\ell_{2}}d^{2,n} & d^{4,n}
	\end{array}\right),
	\end{equation*}
	it follows that for every $i\in\{1,2,3,4\} $, there is $x_{k_{i,j}} $, $j=p_{n-2},\cdots, n-2$ such that:
	\begin{align*}
	d^{i,n}\leq  \sum_{j=p_{n-2}}^{n-2} x_{k_{i,1}}\cdots x_{k_{i,j}} \leq 1
	\end{align*}
	This proves the lemma.\end{proof}
\begin{prop}\label{bounded geometry Proposition} When $(\ell_1,\ell_2)\in [2, \infty)^2\setminus \{(2,2)\} $, then $\overline{A}^{\circ n}_{\eg,\ed} $ contracts the Euclidean metric provided $n$ is large enough. The scale of the contraction is bounded away from 1 independently of  $\ell_1$ and  $\ell_2$ and the particular sequence $b_n$, whereas the moment when the contraction starts depends on the upper bound of $b_n$.
\end{prop}
\begin{proof} For fixed $n\in\N$, putting
	\begin{equation*}
	\overline{A}^{\circ n}_{\eg,\ed}=\left(	\begin{array}{cc}
	d^{1,n}(z_1,z_2) &  \frac{\ell_{1}}{\ell_{2}}d^{3,n}(z_1,z_2)\\ \frac{\ell_{2}}{\ell_{1}}d^{2,n}(z_1,z_2) & d^{4,n}(z_1,z_2)
	\end{array}\right);
	\end{equation*}
	with $z_1=1/(\ell_1-1)$ and $z_2=1/(\ell_2-1)$ Thus, $d^{i,n}(z_1,z_2)$,  $i=1,2,3,4 $ are   polynomials of respective degree $n-2$, $n-3$, $n-3$ and  $n-4$; whose  coefficients belong to the interval $[\ell^{-a}, 1-\ell^{-a}] $; with, $\ell=\max\{\ell_1,\ell_2\} $. For fixed $i\in\{1,2,3,4\} $, we denote  by $d^{i,n}_j$ the coefficients  of $d^{i,n}(z_1,z_2)$. Let us put:
	\begin{equation*}
	\begin{cases}
	\begin{array}{ccccc}
	d_1(2)=1 &\mbox{ and }  & d_1(i)=0;\quad i=1,3,4\\
	d_2(3)=1 & \mbox{ and }   & d_2(i)=0;\quad i=1,2,4.
	\end{array}
	\end{cases}
	\end{equation*}
	Then by the \textbf{Lemma~\ref{relatively compact}}, the sums
	\begin{equation*}
	\sum\limits_{j=0}^{p_{n-2}}d^{i,j}z_1^{j+d_1(i)}z_2^{j+d_2(i)}; \quad i=1,2,3,4
	\end{equation*}
	are uniformly bounded. Therefore, for every $i\in\{1,2,3,4\} $
	\begin{equation*}
	\sum\limits_{j=k}^{\infty}d_j^{i,n}z_1^{j+d_1(i)}z_2^{j+d_2(i)}\longrightarrow 0,\quad\mbox{ when } k\longrightarrow\infty
	\end{equation*}
	\begin{lem}
		For every $i\in\{1,2,3,4\} $, the sequence $ (d_j^{i,n})$ tends to zero at least exponentially.
	\end{lem}
	\begin{proof}
		By a simple calculation, we have 
		\begin{equation*}
		d^{1,n}(0,0)=\prod_{i=1}^{p_{n-2}}b_{n-2i+1},\; d^{4,n}(0,0)=\prod_{i=1}^{p_{n-2}}b_{n-2i}, \; 
		d^{2,n}(0,0)=d^{3,n}(0,0)=0.
		\end{equation*}
		Now, suppose that, for given $0<n-1$ and $0<j<n-1$  all the coefficients $d_j^{i,n-1} $  $i\in\{1,2,3,4\} $, tend to zero at least exponentially fast. Then, since 
		\begin{equation*}
		\overline{A}^{\circ n}_{\eg,\ed}=\overline{A}_{\eg,\ed}(n)\overline{A}^{\circ {n-1}}_{\eg,\ed}
		\end{equation*}
		and by the form of the coefficients of $\overline{A}_{\eg,\ed}(n) $, the Lemma is proved and therefore, the \textbf{Proposition~\ref{bounded geometry Proposition}}.
	\end{proof}
\end{proof}

\subsubsection{Particular case of Bounded Geometry}
\begin{prop}
	Let $f\in\mathscr{L} $ with critical exponents $(\ell_1,\ell_2)$ and rotation number $\rho(f)=[abab\cdots]$; for some $a,b\in\R$. If the inequality
	\begin{equation*}\label{limit domain bounded geometry part}
	\begin{array}{r}
	\sqrt{(\ell_1^{-b}-\ell_2^{-a})^2+(t_1(b)t_2(a)+2(\ell_1^{-b}+\ell_2^{-a}))t_1(b)t_2(a)}+\\+t_1(b)t_2(a)+\ell_1^{-b}+\ell_2^{-a}-2 <0;
	\end{array}
	\end{equation*}	
	holds, then the geometry of $f$ is bounded.
\end{prop}
\begin{proof}
	If $\rho(f)=[abab\cdots]$, then 
	\begin{equation*}
	\overline{A}_{\eg,\ed}=\left(	\begin{array}{cc}
	t_1(b) t_2(a)+\ell_1^{-b} & \dfrac{\ell_2 }{\ell_1}\ell_2^{-a}t_2(a) \\\dfrac{\ell_1}{\ell_2 }t_1(b) &\ell_2^{-a} 
	\end{array}\right) 
	\end{equation*}	
	and these eigenvalues $ \lambda_s$ and $ \lambda_u$ are defined as following
	\begin{equation*}
	\begin{array}{r}
	2\lambda_s=-\sqrt{(\ell_1^{-b}-\ell_2^{-a})^2+(t_1(b)t_2(a)+2(\ell_1^{-b}+\ell_2^{-a}))t_1(b)t_2(a)}+\\+t_1(b)t_2(a)+\ell_1^{-b}+\ell_2^{-a}
	\end{array}
	\end{equation*}
	and 
	\begin{equation*}
	\begin{array}{r}
	2\lambda_u=\sqrt{(\ell_1^{-b}-\ell_2^{-a})^2+(t_1(b)t_2(a)+2(\ell_1^{-b}+\ell_2^{-a}))t_1(b)t_2(a)}+\\+t_1(b)t_2(a)+\ell_1^{-b}+\ell_2^{-a};
	\end{array}
	\end{equation*}
	Observe that, $\lambda_s\in(0,1)$. Thus, if $\lambda_u<1$, then $\overline{A}_{\eg,\ed} $ contracts the Euclidean metric. 
	
	This proves the proposition.
\end{proof}
\subsection{Proof of Corollary}	
\begin{lem}\label{wi and anlpha n-1 are comparable}
	Let
	\begin{equation*}
	w_{n}(i)=\dfrac{|(\underline{-q_{n}+(i+1)q_{n-1}},\underline{-q_{n}+iq_{n-1}})|}{|\underline{-q_{n}+iq_{n-1}}|}
	\end{equation*}	 
	be a parameter sequence with, $i=0\cdots a_{n-1}-1$. 	 $w_{n}(i)$,  $i=1\cdots a_{n-1}-1$ and $w^{\ell_{1},\ell_{2} }_{n}(0)$ are comparable to $\alpha_{ n-1} $.
\end{lem}
\begin{proof}
	Suppose that $a_{n-1} >1$ and let $i=1, \cdots, a_{n-1} -1$. We apply the \textbf{Proposition~\ref{koebe principle}} to
	
	\begin{enumerate}
		\item[-] $T=[ \underline{-q_{n}+(i+1)q_{n-1}},\underline{-q_{n}+(i-1)q_{n-1}}]$,
		\item[-] $M=(\underline{-q_{n}+(i+1)q_{n-1}},\underline{-q_{n}+(i-1)q_{n-1}})$,
		\item[-] $S=\underline{-q_{n}+(i+1)q_{n-1}}$, 
		\item[-] $D=\underline{-q_{n}+(i-1)q_{n-1}}$,
		\item[-] $f^{q_{n}-(i-1)q_{n-1}} $.
	\end{enumerate}
	\begin{enumerate}
		\item For every $j<q_{n}-(i-1)q_{n-1} $, $f^j(T)\cap\overline{U}=\emptyset $; so $f^{q_{n}-(i-1)q_{n-1}} $ is diffeomorphism on $T$ (\textbf{Remark~\ref{diffeo on an interval}});
		\item the set $\bigcup_{j=0}^{q_{n}-(i-1)q_{n-1}} (T) $ covers the circle at most two times;
		\item for $n$ large enough, and by the \textbf{Proposition~\ref{preimage and gaps adjacent}}, we have
		\begin{equation*}
		\dfrac{f^{q_{n}-(i-1)q_{n-1}}(M)}{f^{q_{n}-(i-1)q_{n-1}}(S)}  <
		\dfrac{f^{q_{n}-(i-1)q_{n-1}}(M)}{f^{q_{n}-(i-1)q_{n-1}}(D)}
		=	\dfrac{|( \underline{0}, \underline{-2q_{n-1}})|}{|\underline{-2q_{n-1}}|}
		<K.
		\end{equation*}
		Therefore, it follows from the  \textbf{Proposition~\ref{koebe principle}} and \textbf{Proposition~\ref{preimage and gaps adjacent}} $\mbox{(} |\underline{-q_{n-1}}| \mbox{ and } |( \underline{0}, \underline{-q_{n-1}}]| $ are comparable) that $w_{n}(i)$ and $\alpha_{n-1} $ are comparable.
	\end{enumerate}
	For $i=0$ (which is the only case when $a=1$\mbox{)}, we apply the 	\textbf{Proposition~\ref{koebe principle}} to
	\begin{enumerate}
		\item[-] $T=[ \underline{-q_{n}+q_{n-1}+1},\underline{-q_{n}-q_{n-1}+1}]$,
		\item[-] $M=(\underline{-q_{n}+q_{n-1}+1},\underline{-q_{n}-q_{n-1}+1})$,
		\item[-] $S=\underline{-q_{n}+q_{n-1}+1}$, 
		\item[-] $D=\underline{-q_{n}-q_{n-1}+1}$,
		\item[-] $f^{q_{n}-q_{n-1}-1} $.
	\end{enumerate}
	As before, the hypotheses are satisfied. And for $n$ large enough, 	
	\begin{equation*}
	w^{\ell_{1},\ell_{2}}_{n}(0)=\dfrac{|(\underline{-q_{n}+q_{n-1}+1},\underline{-q_{n}+1})|}{|\underline{-q_{n}+1}|};
	\end{equation*}		
	which is also uniformly comparable to $\alpha_{n-1} $. 
	
	This concludes the proof.
\end{proof}	
The rest of the proof of \textbf{Corollary} is as in \cite{8} (Theorem 1.4 and Theorem 1.5).


\begin{thebibliography}{10}\itemsep=2pt
\bibitem{1}  de Melo, W.  and  Van  Strien, S.,
One-Dimensional Dynamics: The Schwarzian Derivative And
Beyond, \textit{Amer. J. Math.},  1988, vol.\,18, no.\,2, pp.\,159--162.
\bibitem{2} de Melo,  W.  and  Van  Strien, S., One-Dimensional Dynamics: The Schwarzian Derivative and Beyond, \textit{Ann. of Math.}, 1989, vol.\,129,  pp.\,519--546.
\bibitem{3}  de Melo, W.  and  Van  Strien, S.,  \textit{One-Dimensional Dynamics}, Springer--Verlag, 1993.
\bibitem{4}{4}  Graczyk, J., Dynamics of circle maps with flat spots,\textit{ Fund. Math.}, 2010, vol.\,209,  no.\,3, pp.\,267--290.
\bibitem{5}  Graczyk, J.,   Jonker,  L.\,B.,   \'{S}wi\k{a}tek, G.,   Tangerman, F.\,M. and  Veerman, J.\, J.\, P.,
Differentiable Circle Maps with a Flat Interval, \textit{Commun. Math. Phys.}, 1995, vol.\,173, no.\,3, pp.\,599--622.
\bibitem{6} Mendes, P., A metric property of Cherry vector fields on the torus, \textit{J. Differential Equations}, 1991, vol.\,89, no.\,2, pp.\,305--316.
\bibitem{7} Martens, M., Palmisano, L.,  Invariant Manifolds for Non-differentiable Operators, \textsf{arXiv:1704.06328}, (20 Apr 2017).
\bibitem{8} Martens,  M.,  Strien, S.,  Melo,  W.  and  Mendes, P., On Cherry flows, \textit{Ergod. Theory Dyn. Syst.}, 1990, vol.\, 10 , pp.\,531--554.
\bibitem{9}  Misiurewicz, M., Rotation interval for a class of maps of the real line into itself,  \textit{Erg. Th. and Dyn. Sys.} 1986, vol.\,6, no.\,3, pp.\,17--132.
\bibitem {10} Moreira, P. C. and Ruas, A. A. Gaspar, Metric properties of {C}herry flows,  \textit{J. Differential Equations}, 1992, vol.\,97,
no.\,1, pp.\,16--26.
\bibitem{11}  Palmisano, L.,   A Phase Transition for circle Maps and Cherry Flows,  \textit{Commun. Math. Phys.}, 2013, vol.\,321, no.\,1, pp.\,135--155.
\bibitem{12} Palmisano, L., On physical measures for Cherry flows, \textit{Fund. Math.}, 2016, vol.\,232, no.\,2, pp.\,167--179.
\bibitem{13}{Palmisano, L.}, Cherry Flows with non-trivial attractors, \textit{Fund. Math.}, 2019 vol.\,244, no.\,3, pp.\,243--253. 
\bibitem{14} Palmisano, L., Quasi-symmetric conjugacy for circle maps with a flat interval, \textit{Ergodic Theory Dynam. Systems}, 2019 vol.\,39, no.\,2, pp.\,425--445.
\bibitem{15}  Palmisano,  L. and  Tangue, B., A Phase Transition for Circle Maps with a Flat Spot and Different Critical Exponents, \textsf{arXiv: 1907.10909v1}, (27 Jul. 2019).
\bibitem{16}  \'{S}wi\k{a}tek, G. ,  Rational rotation numbers for maps of the circle, \textit{Comm. Math. Phys.} 1988, vol.\,119, no.\,1, pp.\,109--128.
\bibitem {17} Tangerman, F.\,M. and Veerman, J.\,J.\,P.,
Scalings in circle maps. {I}, \textit{Comm. Math. Phys.}, 1990, vol.\,134, no.\,1, pp.\,89--107.
\bibitem{18} Tangerman, F. M. and Veerman, J.\,J.\,P.,
Scalings in circle maps. {II}, \textit{Comm. Math. Phys.}, 1991, vol.\,141, no.\,3, pp.\,279--291.
\bibitem{19}	Veerman, J.\,J.\,P.,  Irrational Rotation Numbers, \textit{Nonlinearity}, 1989, vol.\,3, no.\,3, pp.\,419--428.

\end{thebibliography}
\end{document}